\documentclass{amsart}
\usepackage{amssymb,amsmath,amsthm,graphics,amscd,amsfonts}
\usepackage{color}

\theoremstyle{plain}
\newtheorem{theorem}{Theorem}[section]
\newtheorem{proposition}[theorem]{Proposition}

\newtheorem{example}[theorem]{Example}

\newcommand{\bfC}{{\mathbb C}}

\newcommand{\bfR}{{\mathbb R}}

\newcommand{\barj}{{\overline j}}
\newcommand{\bark}{{\overline k}}
\newcommand{\barl}{{\overline \ell}}

\newcommand{\baru}{{\overline u}}

\newcommand{\barv}{{\overline v}}

\newcommand{\barpartial}{{\overline \partial}}

\newcommand{\mapright}[1]{\smash{\mathop{   \hbox to 0.7cm{\rightarrowfill}}
  \limits^{#1}}}

\newcommand{\Ric}{\operatorname{Ric}}

\begin{document}

\title{The weighted Laplacians on real and complex metric measure spaces}
\author{Akito Futaki}
\address{Graduate School of Mathematical Sciences, The University of Tokyo, 3-8-1 Komaba, Meguro-ku, Tokyo 153-8914, Japan}
\email{afutaki@ms.u-tokyo.ac.jp}
\date{April 7, 2014 }
\dedicatory{Dedicated to the memory of Professor Shoshichi Kobayashi}
\begin{abstract} 
In this short note we
compare the weighted Laplacians on real and complex (K\"ahler) metric measure spaces.
In the compact case K\"ahler metric measure spaces are considered on Fano 
manifolds for the study of K\"ahler-Einstein metrics while real metric measure spaces are considered
with Bakry-\'Emery Ricci tensor. There are twisted Laplacians which are useful in both cases
but look alike each other. We see that
if we consider {\it noncompact} complete manifolds significant differences appear.
\end{abstract}

\maketitle

\section{Introduction}
The weighted Laplacians can be considered on smooth metric measure spaces 
using the exterior derivative $d$ on Riemannian manifolds and the $\barpartial$-operator
on complex K\"ahler manifolds. We considered in \cite{futaki87} and \cite{futaki88} the weighted $\barpartial$-Laplacian 
\begin{equation*}
\Delta_F u = \Delta_\barpartial\, u + \nabla_i F\ \nabla^i u = g^{i\barj}\nabla_i\nabla_\barj\, u + \nabla_i F\ \nabla^i u 
\end{equation*}
on a Fano manifold $M$ where $\Delta_F$ acts on {\it complex-valued} smooth functions 
$u \in C_\bfC^\infty(M)$. 
Here, a Fano manifold, by definition, has  positive first Chern class, and we have chosen a K\"ahler form 
\begin{equation*}
\omega = \sqrt{-1}\ g_{i\barj} dz^i \wedge d\overline{z^j}
\end{equation*}
in $2\pi c_1(M)$.  The real-valued smooth function $F$ is chosen so that
\begin{equation*}
\operatorname{Ric}(\omega) - \omega = \sqrt{-1}\ \partial\barpartial F,
\end{equation*}
and this is possible since the Ricci form 
\begin{equation*}
\operatorname{Ric}(\omega) = - \sqrt{-1}\ \partial\barpartial \log\det(g_{i\barj}).
\end{equation*}
also represents $2\pi c_1(M)$.
We say $\lambda$ is an eigenvalue of $\Delta_F$ if $\Delta_F u + \lambda u = 0$ for some nonzero complex-valued function $u$.
 If we use the weighted volume $d\mu = e^F \omega^m$ we have
\begin{eqnarray*}
\int_M g( \barpartial u, \barpartial v) d\mu=-\int_M (\Delta_F u)\barv \ d\mu=-\int_M u(\overline{\Delta_F v})d\mu
\end{eqnarray*} 
where 
$$
g( \barpartial u, \barpartial v) = \nabla_\barj u\nabla^\barj\barv = g^{i\barj} \frac{\partial u}{\partial \overline{z^j}}\frac{\partial \barv}{\partial z^i}.
$$
Thus all the eigenvalues are nonnegative real numbers. 
It is shown that the first non-zero eigenvalue $\lambda_1(\Delta_F)$ satisfies
\begin{equation}\label{5}
\lambda_1(\Delta_F) \ge 1
\end{equation}
and the equality holds if and only if the Lie algebra $\mathfrak h(M)$ of all holomorphic vector fields 
is non-zero. In fact, given a holomorphic vector field $X$, there corresponds an eigenfunction $u$ by
$$
X = \mathrm{grad}^\prime u := g^{i\barj} \frac{\partial u}{\partial \overline{z^j}}  \frac{\partial }{\partial z^i} .
$$
This observation was made before the publication of \cite{futaki83.1}, and
was inspired by Shoshichi Kobayashi's book \cite{kobayashitrnsf}, Chapter III, section 7, ``Conformal changes of the Laplacian''.
 In fact the obstruction 
to the existence of K\"ahler-Einstein metrics
introduced
in \cite{futaki83.1} was found using the above observation combined with an idea suggested by J.L.Kazdan
 \cite{KJ} (see also \cite{KJ-FW1}). 
 See section 2 for more about the application of the above observation to the study of K\"ahler-Einstein
 metrics. 
 The proof of (\ref{5}) follows from the general formula
\begin{eqnarray}\label{1-2}
- \int_M \nabla^i (\Delta_F u)\ \nabla_i\baru\ d\mu = \int_M (|\nabla''\nabla''u|^2 + |\barpartial u|^2 )d\mu
\end{eqnarray}
where
$$ |\nabla''\nabla''u|^2 = \nabla^i\nabla^j u\ \nabla_i\nabla_j \baru = g^{i\bark}g^{j\barl}\ \nabla_\bark \nabla_\barl  u \ \nabla_i\nabla_j \baru.$$
See also \cite{Pali08}, \cite{WZ13}, \cite{LiLong13} for other applications to the study in K\"ahler geometry. 
 
On the other hand a similar idea is commonly used in Riemannian geometry and probability theory with the 
 Bakry-\'Emery Ricci curvature $\Ric+\nabla^2 f$ on the weighted Riemannian manifolds $(M, g, f)$. 
This means that $(M, g)$ is a complete Riemannian manifold with the weighted measure $d\mu=e^{-f}dv$, where 
$dv$ denotes the Riemannian volume measure on $(M, g)$ and $f$ is a real-valued $C^2$-function.
We denote by $C^\infty(M)$ (resp. $C^\infty_0(M)$) the set of real-valued smooth functions 
(reps. with compact support).
For all $u, v\in C^\infty_0(M)$, the following integration by parts formula holds
\begin{eqnarray*}
\int_M g( \nabla u, \nabla v) d\mu=-\int_M (\Delta_f u)vd\mu=-\int_M u(\Delta_f v)d\mu,
\end{eqnarray*}
where $\Delta_f$ is the called the weighted Laplacian with respect to the volume measure $\mu$. More precisely, we have
\begin{eqnarray*}
\Delta_f =\Delta -\nabla f\cdot\nabla.
\end{eqnarray*}
Here we denote by $\Delta$ the $d$-Laplacian: $\Delta = - d^\ast d = g^{ij}\nabla_i\nabla_j$ with respect to real coordinates $(x^1, \cdots, x^n)$. 
In \cite{BE}, Bakry and \'Emery proved that for all $u\in C_0^\infty(M)$,
\begin{eqnarray}
\Delta_f |\nabla u|^2-2\langle \nabla u, \nabla \Delta_f u\rangle=2|\nabla^2 u|^2+2(\Ric+\nabla^2 f)(\nabla u, \nabla u). \label{BWF}
\end{eqnarray}
The formula $(\ref{BWF})$ can be viewed as a natural extension of the Bochner-Weitzenb\"ock formula. The equation (\ref{1-2}) can 
also be derived from a similar Weitzeb\"ock type formula. The probabilistic study of the weighted Laplacian was motivated by
the hypercontractivity of Markov semigroups. In \cite{Gross75} Gross showed that the hypercontractivity holds if and only if the logarithmic
Sobolev inequality holds. Then in \cite{BE} Bakry and \'Emery showed that on a smooth metric measure space the logarithmic Sobolev
inequality holds if the Bakry-\'Emery Ricci tensor is bounded from below by a positive constant, that is if there is a positive
constant $C$ such that 
\begin{equation}\label{1-4}
\Ric + \nabla^2 f \ge C g.
\end{equation}
Note in this case we have
\begin{equation}\label{1-5}
\lambda_1(\Delta_f) \ge C.
\end{equation}
See section 3 more about the Bakry-\'Emery Ricci tensor.

It has been a puzzle (to the author) how the real and complex (Fano) cases are different. One immediate difference is that,
while $\Delta_f$ in the Riemannian case is a real operator, 
$\Delta_F$ in the K\"ahler case is not a real operator unless $F$ is constant. This means  that $\Delta_F u$ is a complex valued function
even if $u$ is real valued. Therefore the eigenfunctions corresponding to nonzero eigenvalues can not be real valued. 

In this paper we see, by comparing with the results of Cheng and Zhou \cite{ChengZhou13} in the real
noncompact case, 
that if we consider noncompact complete manifolds then more significant differences appear between the  real and complex weighted Laplacians. For example, when the first nonzero eigenvalue of the twisted Laplacian
attains the expected lower bound, it is a discrete spectrum in the real case, but it can be an essential spectrum in the complex case. Moreover if the the first nonzero eigenvalue of expected lower bound has multiplicity $k$
then in the real case the Gaussian soliton of dimension $k$ splits off, but this is not the case in the
complex case. See sections 4 and 5 for more detail.

\section{The case of Fano manifolds }

Let $M$ have positive first Chern class, i.e. the first Chern class $c_1(M)$ contains a positive closed $(1,1)$-form.
This is equivalent to say the anticanonical bundle of $M$ is ample, and such a manifold is called a Fano manifold.
We choose a K\"ahler from 
\begin{equation}\label{omega}
\omega = \sqrt{-1}\ g_{i\barj}\ dz^i \wedge d\overline{z^j}
\end{equation}
in $2\pi c_1(M)$.
Since the Ricci form 
\begin{equation}\label{Ric}
\operatorname{Ric}(\omega) = - \sqrt{-1} \partial\barpartial \log\det(g_{i\barj}).
\end{equation}
also represents $2\pi c_1(M)$ there exists a 
smooth function $F$ such that
\begin{equation}\label{Fun}
\operatorname{Ric}(\omega) - \omega = \sqrt{-1}\ \partial\barpartial F.
\end{equation}
Denote by $C^\infty_\bfC (M)$ the set of all complex valued functions on $M$. 
For $u$ and $v$ in $C^\infty_\bfC (M)$, we consider the $L^2$-inner product
with respect to the weighted volume $d\mu = e^F \omega^m$
$$ (u,v)_F = \int_M u\barv\,e^F\,\omega^m$$
and
$$ (\barpartial u,\barpartial v)_F = \int_M g(\barpartial u, \barpartial v)\,e^F\,\omega^m.$$
We considered in \cite{futaki87} and \cite{futaki88} the weighted $\barpartial$-Laplacian 
\begin{equation}\label{Delta_F}
\Delta_F u = \Delta_\barpartial\, u + \nabla_i F\ \nabla^i u = g^{i\barj}\nabla_i\nabla_\barj\, u + \nabla_i F\ \nabla^i u.
\end{equation}
We then have
\begin{eqnarray}\label{self}
(\barpartial u,\barpartial v)_F = - (\Delta_F u, v)_F = - (u, \Delta_F v)_F.
\end{eqnarray} 
We say $\lambda$ is an eigenvalue of $\Delta_F$ if $\Delta_F u + \lambda u = 0$ for some nonzero complex-valued function $u$.
Since $\Delta_F$ is self-adjoint by (\ref{self}) all the eigenvalues are nonnegative real numbers. 

\begin{theorem}[c.f. \cite{futaki87}, \cite{futaki88}]\label{F}
The first non-zero eigenvalue $\lambda_1(\Delta_F)$ of $\Delta_F$ satisfies
\begin{equation}\label{1-1}
\lambda_1(\Delta_F) \ge 1
\end{equation}
and the equality holds if and only if the Lie algebra $\mathfrak h(M)$ of all holomorphic vector fields 
is non-zero. In fact, for a holomorphic vector field $X$, there corresponds an eigenfunction $u$ by
$$
X = \mathrm{grad}^\prime u := g^{i\barj} \frac{\partial u}{\partial \overline{z^j}}  \frac{\partial }{\partial z^i}.
$$
\end{theorem}
If we pick another K\"ahler form $\widetilde \omega = \omega + \sqrt{-1} \partial \barpartial \varphi$ with $\varphi \in C^\infty(M)$,
then $\tilde u := u + u^i\varphi_i$ is the first eigenfunction corresponding to the holomorphic vector field $X$ in the previous theorem.
Using this I proved the following theorem in the first version of \cite{futaki83.1}.
\begin{theorem}[\cite{futaki83.1}]\label{F2}
On a Fano manifold if we define $f : \mathfrak h(M) \to \bfC$ by
$$ f(X) = \int_M XF\ \omega^m $$
then $f$ is indecent of the choice of $\omega \in 2\pi c_1(M)$. In particular if $f \ne 0$ then $M$ does not admit a K\"ahler-Einstein metric.
\end{theorem}
If we use the above correspondence in Theorem \ref{F} 
we have a complex valued smooth function $u$ such that $X = \mathrm{grad}^\prime u$ and that
$\Delta_F u + u = 0$. It follows that $XF = -\Delta_\barpartial\,u - u$ and that
\begin{equation}\label{moment}
f(X) = - \int_M u\ \omega^m.
\end{equation}
This implies the following result which is known as Mabuchi's theorem \cite{Mabuchi87-1} for toric Fano manifolds. But this is generally true even if 
$M$ is not toric.
\begin{theorem}
 The character $f$ vanishes if and only if, for the action of the maximal torus of the reductive part of the automorphism group,  
 the barycenter of the moment map lies at $0$. 
 \end{theorem}
 \begin{proof} We have the Chevalley decomposition 
 $\mathfrak h(M) = \mathfrak h_r \oplus \mathfrak h_u$ 
 where $ \mathfrak h_r$ is the maximal reductive sub algebra and $ \mathfrak h_u$ is the unipotent radical. But by
 \cite{Mabuchi90-2} $f$ vanishes on $\mathfrak h_u$. It follows that $f$ vanishes if and only if $f$ vanishes on $\mathfrak h_r$.
 But any element of the automorphism in the reductive group is contained in a maximal torus. Therefore $f$ vanishes if and only
 if it vanishes on the maximal abelian subalgebra. Thus the theorem follows from (\ref{moment}). 
 \end{proof}
 Note also in passing that the theorem of Matsushima \cite{matsushima57}, saying that 
 the Lie algebra $\mathfrak h(M)$ is reductive when $M$ is a K\"ahler-Einstein manifold, follows from the 
 Theorem \ref{F} since if $(M,g)$ is K\"ahler-Einsten then we can take $F = 0$ and $\Delta_F = \Delta_\barpartial$ is a
 real operator so that $\mathfrak h(M)$ is a complexification of the purely imaginary eigenfunctions, the gradient vectors of which
 are Killing vector fields. 

\section{Hypercontractivity, logarithmic Sobolev inequality and Bakry-\'Emery-Ricci tensor}
In this section we review the historical background of Bakry-\'Emery Ricci tensor in probability theory.
We refer the reader to the lecture notes \cite{Guionnet-Zegarlinski} for more details.
A Markov semigroup $P_t = e^{t\mathcal L}$, $t>0$, is said to have hypercontractivity if
$$ ||P_t u||_{L^p(\mu)} \le ||u||_{L^q(\mu)}$$
where $p = p(t) = 1 + (q-1)e^{\frac{2t}{c}} > 1$, $c \in [1,\infty)$ is a constant and $\mu$ is a 
probability measure. 

In \cite{Gross75}, L.Gross showed that $P_t = e^{t\mathcal L}$ satisfies hypercontractivity if 
and only if the following logarithmic Sobolev inequality is satisfied:
$$
\mu(u^2 \log u^2) \le c \mu(u(-\mathcal Lu)) + \mu(u^2)\log \mu(u^2)
$$
where $c \in (0,\infty)$ is a constant independent of $u$.

Let $(M,g, e^{-f}dV_g)$ be a smooth metric measure space. This means that $(M,g)$ is a
Riemannian manifold with the Riemannian volume element $dV_g$, that $f$ is a smooth
function on $M$, and that we consider the twisted volume element $e^{-f}dV_g$.
In \cite{BE}, Bakry and \'Emery showed that if a metric measure space $(M,g, e^{-f}dV_g)$
has finite measure $\int_M e^{-f}dV_g < \infty$, and $\mathrm{Ric}_f := \mathrm{Ric} + \nabla^2 f 
\ge \lambda g$ with a positive constant $\lambda > 0$, then the logarithmic Sobolev inequality
holds with respect to the measure $d\mu = e^{-f}dV_g$
and $\mathcal L u = \Delta_f u = \Delta u - \nabla^i f\nabla_i u$. In particular, 
$P_t = e^{t\mathcal L}$ satisfies hypercontractivity.
Note that Morgan \cite{Morgan05} later proved that if $\mathrm{Ric}_f 
\ge \lambda g$ for some positive constant $\lambda > 0$ then the weighted measure is finite:
 $\int_M e^{-f}dV_g < \infty$. 

The key observation of Bakry and \'Emery is the following Weitzenb\"ock type formula.
\begin{equation}\label{Weitzenbock}
\Delta_f |\nabla u|^2 - 2(\nabla u, \nabla \Delta_f u) = 2|\nabla^2 u|^2 + 2(\mathrm{Ric} + \nabla^2 f)
(\nabla u, \nabla u).
\end{equation}
As an immediate consequence of (\ref{Weitzenbock}) we obtain the following spectral gap.
\begin{proposition}[\cite{BE}, \cite{HeinNaber13}, \cite{Morgan05}; see also \cite{ChengZhou13}]\label{eigen}
In the situation as above, if a non constant function $u \in L^2(e^{-f}dV_g)$ satisfies 
$\Delta_f u + \mu u = 0$ then $\mu \ge \lambda$. That is, the first nonzero eigenvalue 
$\lambda_1(\Delta_f)$ satisfies $\lambda_1(\Delta_f) \ge \lambda$.

\end{proposition}

\section{Spectrum of complete smooth metric measure spaces}

Given a smooth metric measure space $(M,g, e^{-f}dV_g)$ we set $L^2_f(M)$ to be the closure of 
the set $C_c^\infty(M)$ of real-valued smooth functions on $M$ with compact support
with respect to $L^2(d\mu)$-norm, and $H^1_f(M)$ to be the closure of 
the set $C_c^\infty(M)$ of real-valued smooth functions $u$ on $M$ with compact support
with respect to the norm 
$$||u||_{H^1_f} = \left( \int_M (u^2 + |\nabla u|^2) d\mu \right)^\frac12.$$

By the result of Bakry and \'Emery, if $\mathrm{Ric}_f 
\ge \lambda g$ for some positive constant $\lambda > 0$ then we have the logarithmic Sobolev inequality
$$ \int_M u^2 \log u^2 d\mu \le C \int_M |\nabla u|^2 d\mu$$
for $u \in H^1_f(M)$ with $\int_M u^2 d\mu = 1$.

It is known (c.f. \cite{HeinNaber13}, \cite{ChengZhou13}) that if the logarithmic Sobolev
inequality holds then we have the compact embedding $H^1_f(M) \hookrightarrow
L^2_f(M)$, and the spectrum of the twisted Laplacian $\Delta_f$ is discrete.
We denote by $\lambda_1(\Delta_f)$ the first nonzero eigenvalue of $\Delta_f$ with eigenfunction in
$H^1_f(M)$. 

Now we recall a recent result of Cheng and Zhou \cite{ChengZhou13}. 
First let us see two examples.
\begin{example}
Let $(\bfR^n, g_{\mathrm{can}}, \frac{|x|^2}4)$ be the Gaussian soliton, that is $g_{\mathrm{can}}$ 
is the flat metric $\frac 12 \nabla^2 |x|^2$ with $f = \frac{|x|^2}4$ so that
 $\mathrm{Ric}_f = 0 + \frac 12 \nabla x \cdot \nabla x$. Then we have $\lambda_1(\Delta_f) = \frac 12$.
This is because Proposition \ref{eigen} shows $\lambda_1(\Delta_f) \ge \frac12$. But the equality
is attained by $u = x^1$ since $\Delta u + \nabla u\cdot \nabla f = dx^1\cdot (\frac 12 \sum x^idx^i) = \frac{x^1}2 = \frac u2$.
\end{example}
\begin{example} Consider $S^{n-k}(\sqrt{2(n-k-1)}) \times \bfR^k$ for $n-k \ge 2$ and $k\ge 1$.
Take $f = \frac {|t|^2}4$ with $t \in \bfR^k$. Then we have $\mathrm{Ric}_f = \frac12 g$ and 
$\lambda_1(\Delta_f) = \frac12$, and the corresponding eigenfunctions are linear functions
in $\bfR^k$. 
\end{example}

\begin{theorem}[\cite{ChengZhou13}]\label{ChengZhou}
Let $(M^n, g, e^{-f}dV_g)$ be a complete smooth metric measure space with $\mathrm{Ric}_f \ge \lambda g$ for a positive constant $\lambda$. If $\lambda_1(\Delta_f) = \lambda$ with multiplicity $k$ then 
$M$ is isometric to $\Sigma^{n-k} \times \bfR^k$ with $\lambda_1(\Delta_f^{\Sigma}) > \lambda$
and $f$ is written in the form $f(p,t) = f(p,0) + \frac{\lambda}2 |t|^2$ where $p \in \Sigma$ and $t \in \bfR^k$. 
\end{theorem}

In the case of shrinking Ricci soliton
\begin{equation}\label{shrinking}
\mathrm{Ric}_f = \lambda g, 
\end{equation}
it is shown in \cite{ChengZhou13} that we have $\lambda \le \lambda_1(\Delta_f) \le 2\lambda$.
This is because we always have 
\begin{equation}\label{2lambda}
\Delta_f f + 2\lambda f = 0,
\end{equation}
and $f$ belongs to $H^1_f(M)$. In the case of a compact shrinking soliton, a lower bound of
$\lambda_1(\Delta_f)$ can be given in terms of the diameter, and combined with
(\ref{2lambda}) we can obtain a universal lower bound of the diameter (\cite{FS13}, \cite{FLL}).

\section{Complete K\"ahler metric measure spaces}
We say that $(M,g,e^F dV_g)$ is a complete K\"ahler metric measure space if $(M,g)$ is a
complete K\"ahler manifold with 
\begin{equation}\label{mmK}
\operatorname{Ric}(\omega) - \omega = \sqrt{-1}\ \partial\barpartial F,
\end{equation}
where 
\begin{equation}\label{omega-c}
\omega = \sqrt{-1}\ g_{i\barj} dz^i \wedge d\overline{z^j}
\end{equation}
is the K\"ahler form and $F$ is a smooth function. If $M$ is compact $M$ is naturally a Fano
manifold. Note also $(M,g)$ is a gradient K\"ahler-Ricci soliton if in addition
$$ \nabla'' \nabla'' F = 0.$$
By the same arguments of Morgan \cite{Morgan05} we can show that the weighted volume
$\int_M e^F dV_g$ is finite. We consider the same weighted Laplacian $\Delta_F$ as in the 
Fano case in section 2. Namely, $\Delta_F$ is given by the same formula (\ref{Delta_F}) and 
acts on the complex-valued functions $C^\infty_\bfC(M)$. 

\begin{theorem}\label{cK} 
Let $(M,g,e^F dV_g)$ be a complete K\"ahler metric measure space. Then $\lambda_1(\Delta_F) \ge
1$ and there is an imbedding of the $1$-eigenspace $\Lambda_1$ to the Lie algebra $\mathfrak h(M)$ of
all holomorphic vector fields on $M$.
\end{theorem}

Unlike the real case in section 3, $\Lambda_1$ can be infinite dimensional so that $1$ is an
essential spectrum. For example, consider $(\bfC^n, g_{\mathrm{can}}, e^{-|z|^2}dV_{g_{\mathrm{can}}})$
with $n \ge 2$ 
where $g_{\mathrm{can}} = \nabla z\cdot \nabla\overline{z}$. With $F = -|z|^2$,
$(\bfC^n, g_{\mathrm{can}}, e^{-|z|^2}dV_{g_{\mathrm{can}}})$ is a complete K\"ahler metric measure
space, or even gradient shrinking K\"ahler-Ricci soliton. Let $v(z_2, \cdots, v_n)$ be polynomials
in $z_2, \cdots, z_n$, and put $ u = v\overline{z_1}$. Then we see
$$ \Delta_F u = -u.$$
The space of all polynomials is of course infinite dimensional. Note that the imbedding in Theorem \ref{cK}
is not surjective since there are non-integrable holomorphic functions in $z_2, \cdots, z_n$. 

Note also a result of the same type as Theorem \ref{ChengZhou} also does not hold in K\"ahler situation as the following example shows. 
Let $\Sigma$ be an $n$-dimensional Fano manifold with $\dim \mathfrak h(\Sigma) = k$. Then there is no splitting of Euclidean factor. 
Moreover, as seen above, the first eigenspace of the Euclidean factor $\bfC^k$ is infinite dimensional.

Note that in the case of K\"ahler-Ricci solitons the logarithmic Sobolev inequality holds for real-valued
functions (with respect to weighted $d$-Laplacian). Similar arguments were given in Gross \cite{Gross99}
and Gross and Qian \cite{GrossQian04}.

\bibliographystyle{amsalpha}

\end{document}